\documentclass[11pt]{amsart}
\usepackage{amsmath,amsfonts,amscd,amssymb}
\usepackage[margin=3.4cm]{geometry}
\theoremstyle{plain}
\newcounter{thmcount}
\newcommand{\Z}{\mathbb{Z}}
\newcommand{\Q}{\mathbb{Q}}

\newtheorem{theorem}[thmcount]{Theorem}

\newtheorem{proposition}[thmcount]{Proposition}
\newtheorem{lemma}[thmcount]{Lemma}

\theoremstyle{definition}
\newtheorem{remark}[thmcount]{Remark}
\newtheorem*{remark*}{Remark}

\usepackage[utf8]{inputenc}
\usepackage{amssymb}
\usepackage[english]{babel}

\usepackage[OT2,T1]{fontenc}
\DeclareSymbolFont{cyrletters}{OT2}{wncyr}{m}{n}
\DeclareMathSymbol{\Sha}{\mathalpha}{cyrletters}{"58}

\usepackage{amsthm}
\usepackage{mathtools}
\usepackage{dirtytalk}
\usepackage[alphabetic,lite]{amsrefs}
\usepackage{tikz-cd}
\usepackage{textcomp}

\usepackage{hyperref}

\mathtoolsset{showonlyrefs}
\def\Q{{\mathbb Q}}

\title{Non-isomorphic abelian varieties with the same arithmetic}

\subjclass[2020]{Primary 11G10; Secondary 14K02.}

\author{Jamie Bell}
\address{University College London, Gower Street, London, WC1E 6BT, UK}
\email{james.bell.20@ucl.ac.uk}
\address{ORCID 0000-0002-6361-2237}

\begin{document}

\begin{abstract}
    We construct two abelian varieties over $\mathbb{Q}$ which are not isomorphic, but have isomorphic Mordell--Weil groups over every number field, isomorphic Tate modules and equal values for several other invariants.
\end{abstract}

\maketitle

\section{Introduction}
The aim of this paper is to prove the following theorem.

\begin{theorem}\label{maintheorem}
    There exist abelian varieties $A$ and $B$ defined over $\Q$ which are not isomorphic to each other but satisfy the following, over every number field $F$:
    \begin{itemize}
        \item The Mordell--Weil groups $A(F)$ and $B(F)$ are isomorphic.
        \item The $n$-Selmer groups of $A$ and $B$ are isomorphic, for every positive integer $n$.
        \item The Tamagawa numbers $c_v(A)$ and $c_v(B)$ are equal, for every place $v$.
        \item The Tate--Shafarevich groups $\Sha(A)$ and $\Sha(B)$ are isomorphic.
        \item The $L$-functions $L(A/F,s)$ and $L(B/F,s)$ are equal.
        \item The conductors of $A$ and $B$ are equal.
        \item The regulators $\mathrm{Reg}(A/F)$ and $ \mathrm{Reg}(B/F)$ are equal.
        \item For every prime $\ell$, the Tate modules $T_{\ell}(A)$ and $T_{\ell}(B)$ are isomorphic.
    \end{itemize}
\end{theorem}
In other words, if we wish to distinguish abelian varieties by their arithmetic properties, this list is insufficient.

Mazur, Rubin \cite{MR} and Chiu \cite{Chiu} have considered the related problem of finding which properties force elliptic curves to be isogenous. If $E$ and $E'$ are elliptic curves defined over a number field $K$, and $\mathrm{Sel}_p(E/L)=\mathrm{Sel}_p(E'/L)$ for all finite extensions $L/K$ and all but finitely many primes $p$, then $E$ and $E'$ are isogenous (\cite{Chiu} Thm. 1.8). However looking at a single $n$, even when we also consider the $n$-Selmer groups of the quadratic twists of $E$ and $E'$, is not enough to tell us that they are isogenous \cite{MR}. Chiu used the work of Faltings, who showed that that if $A$ and $A'$ are abelian varieties over a number field $K$, with Tate modules satisfying $T_{\ell}(A) \otimes_{\Z_{\ell}}\Q_{\ell} \cong T_{\ell}(A') \otimes_{\Z_{\ell}}\Q_{\ell}$ for some prime $\ell$, then $A$ and $A'$ are isogenous over $K$. Similarly if the local factors in their $L$-functions satisfy $L_v(A/K,s)=L_v(A'/K,s)$ for all but finitely many places $v$ of $K$, then $A$ and $A'$ are isogenous (\cite{Faltings} \S 5, Cor. 2).

\subsection*{Notation}

Let $[n]$ denote the multiplication by $n$ isogeny on an abelian variety, or the multiplication by $n$ map on an abelian group. Let $G[n]$ be the kernel of $[n]$ on $G$.

For a field $k$, let $\bar{k}$ denote its algebraic closure, and $G_k$ the Galois group of $\bar{k}/k$.

\section{Properties of \texorpdfstring{$A$ and $B$}{A and B}}
\begin{proposition}\label{mainproposition}
    Suppose $A$ and $B$ are abelian varieties over a number field $k$, and that there exist isogenies from $A$ to $B$ of degree coprime to $\ell$, for all primes $\ell$. Then $A$ and $B$ have the same properties as listed in the statement of Theorem \ref{maintheorem}, for all number fields $F$ containing $k$.
\end{proposition}

\begin{lemma}\label{lemma1}
    Suppose $f$ is a functor from abelian varieties over a fixed field $k$ to the category of abelian groups $G$ with $G[n]$ and $G/nG$ finite for all positive integers $n$. Suppose also that, for every abelian variety $A$ and every positive integer $n$, the map $f([n]): f(A) \rightarrow f(A)$ is multiplication by $n$. Then for any isogeny $\phi$ of degree coprime to $\ell$, $|\mathrm{ker}(f(\phi))|$ and $|\mathrm{coker}(f(\phi))|$ are finite and coprime to $\ell$.
\end{lemma}
\begin{proof}
    This follows from the existence of conjugate isogenies. Given $\phi : X \rightarrow Y$, there exists $\phi': Y \rightarrow X$ such that $\phi' \circ \phi = [\mathrm{deg}(\phi)]$ on $X$ and $\phi \circ \phi' = [\mathrm{deg}(\phi)]$ on $Y$.

    Now $f(\phi') \circ f(\phi) = [\mathrm{deg}(\phi)]$, so $\mathrm{ker}(f(\phi)) \le \mathrm{ker}([\mathrm{deg}(\phi)])$. Because $|\mathrm{ker}([\mathrm{deg}(\phi)])|$ is coprime to $\ell$, the kernel has the required property. Similarly $f(\phi) \circ f(\phi') = [\mathrm{deg}(\phi)]$, so $\mathrm{im}(f(\phi)) \ge \mathrm{im}([\mathrm{deg}(\phi)])$ and $\mathrm{coker}(f(\phi))$ is a quotient of $\mathrm{coker}([\mathrm{deg}(\phi)])$. The conclusion follows.
\end{proof}

\begin{lemma}\label{lemma2}
    Suppose $f$ is as in Lemma \ref{lemma1}, and maps to finite groups. Suppose there exist isogenies from $A$ to $B$ of degree coprime to $\ell$, for all primes $\ell$. Then $f(A) \cong f(B)$.
\end{lemma}

\begin{proof}
    We first prove that $|f(A)|=|f(B)|$. Suppose $\phi : A \rightarrow B$ is an isogeny. Then $$\frac{|f(A)|}{|f(B)|} = \frac{|\mathrm{ker}(f(\phi))|}{|\mathrm{coker}(f(\phi))|}.$$
    By Lemma \ref{lemma1}, if we pick $\phi$ of degree coprime to a prime $\ell$, the right hand side has $\ell$-adic valuation 0. Doing this for a range of isogenies, we see that it equals 1, so $|f(A)|=|f(B)|$.

    Now consider the functor $X \rightarrow f(X)[n]$ for some integer $n$. This meets the required conditions, so $|f(A)[n]| = |f(B)[n]|$ for all $n$. By the structure theorem for finite abelian groups, this is enough to show $f(A) \cong f(B)$.
\end{proof}

\begin{remark}\label{remark3}
    The same holds if $f$ maps to finitely generated abelian groups. The groups $f(A)$ and $f(B)$ must have the same rank, as the cokernels of the maps between them are finite. Then we can apply the lemma to the torsion parts.
\end{remark}

\begin{proof}[Proof of Proposition \ref{mainproposition}]
    The Mordell--Weil groups and $n$-Selmer groups are isomorphic by a direct application of Lemma \ref{lemma2} and Remark \ref{remark3}, and so are the Tamagawa numbers as $c_v(A) = |A(k_v)/A_0(k_v)|$. The Tate--Shafarevich groups are isomorphic as they are determined by the finite groups $\Sha[\ell^n]$ for primes $\ell$, and we can apply Lemma \ref{lemma2} to these. The equality of the $L$-functions and conductors follows from the existence of an isogeny $A \rightarrow B$. 
    
    For the regulators, we will prove that given an isogeny $\phi: A \rightarrow B$, we have \[\frac{\mathrm{Reg}(A/k)}{\mathrm{Reg}(B/k)} = \frac{|\mathrm{coker}(\phi(A(k)/A(k)_{\mathrm{tors}}))|}{|\mathrm{coker}(\hat{\phi}(\hat{B}(k)/\hat{B}(k)_{\mathrm{tors}}))|},\] where $\hat{B}$ and $\hat{\phi}$ are the duals of $B$ and $\phi$. By picking isogenies, we can then show that the right hand side is coprime to any prime by Lemma \ref{lemma1} and hence the regulators are equal. To do this, recall that the regulator is defined as $|\mathrm{det}\langle a_i, \hat{a}_j \rangle|$, where $\{a_i\}$ is a lattice basis for $A(k)/A(k)_{\mathrm{tors}}$ and $\{\hat{a}_j\}$ for $\hat{A}(k)/\hat{A}(k)_{\mathrm{tors}}$, and $\langle \_, \_ \rangle$ is the height pairing. We will define $\{b_i\}$ and $\{\hat{b}_j\}$ similarly. As in (\cite{MilneADT} proof of Theorem I.7.3), functoriality of the height pairing implies that $|\mathrm{det}\langle a_i, \hat{\phi}(\hat{b}_j)\rangle| = |\mathrm{det}\langle \phi(a_i), \hat{b}_j\rangle|$. By multilinearity of the determinant and height pairing, these determinants differ from the regulators of $A$ and $B$ by factors of $|\mathrm{coker}(\hat{\phi}(\hat{B}(k)/\hat{B}(k)_{\mathrm{tors}}))|$ and $|\mathrm{coker}(\phi(A(k)/A(k)_{\mathrm{tors}}))|$ respectively and the result follows.
    
    Finally for the Tate modules $T_{\ell}(A)$ and $T_{\ell}(B)$, pick an isogeny $\phi$ of degree coprime to $\ell$. The map $[\mathrm{deg}(\phi)]$ is an isomorphism on $T_{\ell}(A)$ and $T_{\ell}(B)$, so the proof of Lemma \ref{lemma1} implies that $\phi$ induces an isomorphism of Tate modules as groups. Because $\phi$ commutes with the action of $G_k$ on points, it does on the Tate module also, so they are isomorphic as $\Z[G_k]$-modules.
 
\end{proof}

\section{Existence}
\begin{theorem}\label{secondtheorem}
    There exist abelian varieties $A$ and $B$ over $\Q$ which are not isomorphic over $\Q$, but for any prime $\ell$ there exists an isogeny between them of degree coprime to $\ell$.
\end{theorem}

Combined with Proposition \ref{mainproposition}, this proves Theorem \ref{maintheorem}. We will do this by considering $\Z[G]$-modules, as done by Milne in \cite{Milne}.

Let $p$ be a prime, $K$ the $p^{th}$ cyclotomic field, and $\mathcal{O}_K = \Z[\zeta]$ its ring of integers, where $\zeta$ is a $p^{th}$ root of unity. Let $G$ be the group $C_p$, generated by an element $g$. Note that an ideal in $\mathcal{O}_K$ is a $\Z[G]$-module, with $g$ acting as multiplication by $\zeta$.

\begin{lemma}\label{lemma7}
    Two ideals in $\mathcal{O}_K$ are isomorphic as $\Z[G]$-modules if and only if they are in the same ideal class.
\end{lemma}
\begin{proof}
    See Curtis--Reiner (\cite{CR} \S 74, p. 507).
\end{proof}

\begin{lemma}\label{lemma8}
    Let $M$ and $N$ be ideals in $\mathcal{O}_K$. Then $M \otimes_{\Z}\Z_{\ell} \cong N \otimes_{\Z} \Z_{\ell}$ for all primes $\ell$.
\end{lemma}
\begin{proof}
     $M \otimes_{\Z} \Z_{\ell}$ is a $\Z_{\ell}[G]$-module, and in fact it is an ideal in $\Z[\zeta] \otimes_{\Z} \Z_{\ell} \cong \frac{\Z_{\ell}[X]}{(1+X+ \ldots +X^{p-1})}$. This cyclotomic polynomial factorises into distinct irreducible factors $P_1, \ldots, P_t$ over $\Z_{\ell}$. The only prime that divides the discriminant of $1 + X + \ldots + X^{p-1}$ is $p$, so for $\ell \neq p$ the polynomials $P_1, \ldots, P_t$ are coprime. If $\ell=p$, then $1 + X + \ldots + X^{p-1}$ is irreducible so $t=1$. Therefore in either case we have \[\Z[\zeta] \otimes_{\Z} \Z_{\ell} \cong \frac{\Z_{\ell}[X]}{P_1(X)} \times \ldots \times \frac{\Z_{\ell}[X]}{P_t(X)}.\] Therefore the ideal $M \otimes_{\Z} \Z_{\ell}$ is a product of ideals $M_i$ in $\frac{\Z_{\ell}[X]}{P_i(X)}$. Considering the $\Z_{\ell}$-rank of these tells us that $M_i$ is never the zero ideal, as \[p-1 = \mathrm{rk}_{\Z}(M) = \mathrm{rk}_{\Z_{\ell}}(M \otimes_{\Z} \Z_{\ell}) = \sum_i \mathrm{rk}_{\Z_{\ell}}(M_i) \le \sum_i \mathrm{rk}_{\Z_{\ell}}\left( \frac{\Z_{\ell}[X]}{P_i(X)} \right)=p-1.\] Each ring $\frac{\Z_{\ell}[X]}{P_i(X)}$ is the ring of integers of the cyclotomic extension $\frac{\Q_{\ell}[X]}{P_i(X)}$ (\cite{Serre} Ch. IV, \S 4, Prop. 16 and 17), a local field, so is a principal ideal domain. Therefore $M_i$ is a non-zero principal ideal so is isomorphic to $\frac{\Z_{\ell}[X]}{P_i(X)}$ as $\Z_{\ell}[G]$-modules, so 
     $M \otimes_{\Z} \Z_{\ell} \cong \Z[\zeta] \otimes_{\Z} \Z_{\ell}$. The same is true for $N \otimes_{\Z} \Z_{\ell}$ so the result follows.
\end{proof}

For the construction of abelian varieties from these $\Z[G]$-modules, we follow Milne (\cite{Milne} \S 2), and use the notation of that chapter. Suppose $M$ and $N$ are ideals in $\mathcal{O}_K$. Given an abelian variety $A$ defined over $\Q$, we can construct two abelian varieties that are isomorphic over $\bar{\Q}$ to $A^{p-1}$, and isogenous to each other over $\Q$. Denote these by $M \otimes A$ and $N \otimes A$ as in Milne (\cite{Milne} \S 2), considering $M$ and $N$ as $R := \Z$-modules.

\begin{lemma}
    [= \cite{Milne} Prop. 6(a)]\label{milne} Suppose $M$ and $N$ are ideals in $\mathcal{O}_K$, and $A$ is an abelian variety over a number field $k$. Suppose $G_k$ has a quotient isomorphic to $G$, and view $M$ and $N$ as $\Z[G_k]$-modules via the action of $G$. Suppose that $\phi: M \rightarrow N$ is a module homomorphism with finite cokernel. Then $\phi_A: M \otimes A \rightarrow N \otimes A$ is an isogeny defined over $k$, and its degree is $|\mathrm{coker}(\phi)|^{2 \mathrm{dim}(A)}$.
\end{lemma}

The following result is a partial converse to this lemma.

\begin{lemma}\label{milneconverse}
    Suppose $A/k$ is an abelian variety with $\mathrm{End}_{\bar{k}}(A) \cong \Z$. Then if $M$ and $N$ are ideals in $\mathcal{O}_K$, viewed as $\Z[G_k]$-modules as in Lemma \ref{milne}, and $M \otimes A$ is isomorphic to $N \otimes A$ over $k$, then $M$ and $N$ are isomorphic as $\Z[G_k]$-modules.
\end{lemma}
\begin{proof}
        First let us fix some notation. Let $n:= p-1$. As in (\cite{Milne} \S 2), we have isomorphisms $\psi_M : \Z^n \rightarrow M$ and $\psi_{M \otimes A}: (A_{\bar{k}})^n \rightarrow (M \otimes A)_{\bar{k}}$, and similarly for $N$. Denote a choice of isomorphism $\mathrm{Aut}(\Z^n) \rightarrow \mathrm{Aut}_{\bar{k}}(A^n)$, as used in (\cite{Milne} \S 2) in the construction of $M \otimes A$ and $N \otimes A$, by $\rho$. For now, let the action of $\sigma \in G_k$ on a module or variety $X$ be written as $\chi_X(\sigma)$, and then define cocycles from $G_k$ to $\mathrm{Aut}(\Z^n)$ and $\mathrm{Aut}_{\bar{k}}(A^n)$ by
    \begin{eqnarray}
        s_M(\sigma) &=& \psi_M^{-1} \circ \chi_M(\sigma) \circ \psi_M \circ \chi_{\Z^n}(\sigma^{-1})\\
        s_{M\otimes A}(\sigma) &=& \psi_{M \otimes A}^{-1} \circ \chi_{(M \otimes A)_{\bar{k}}}(\sigma) \circ \psi_{M \otimes A} \circ\chi_{(A_{\bar{k}})^n}(\sigma^{-1})
    \end{eqnarray}
    and similarly for $N$. Note that we view $\Z^n$ as a trivial Galois module. By construction, $\rho(s_M)=s_{M \otimes A}$ and likewise for $N$. Henceforth, we will drop the notation $\chi_X$.

    Now suppose there is an isomorphism $\phi_A : M \otimes A \rightarrow N \otimes A$ defined over $k$. We will reverse Milne's construction, and show there is a $\Z[G_k]$-module isomorphism $\phi:M \rightarrow N$. Note that as $\phi_A$ is an isomorphism, $\psi_{N \otimes A}^{-1} \phi_A \psi_{M \otimes A}$ is an automorphism of $A^n$, so we can define $\phi$ by $\psi_N^{-1}\phi\psi_M = \rho^{-1}(\psi_{N \otimes A}^{-1} \phi_A\psi_{M \otimes A})$. As $\psi_N^{-1}\phi\psi_M$ is an isomorphism, so is $\phi$. 

    The fact that $\phi_A$ is defined over $k$ is equivalent to the fact that $\sigma \phi_A = \phi_A \sigma$ for all $\sigma \in G_k$. We shall prove the equivalent property for $\phi$, which implies $\phi$ is an isomorphism of $\Z[G_k]$-modules.

    Note that in our case $\psi_{N \otimes A}^{-1} \phi_A \psi_{M \otimes A}$ is an automorphism of $A^n$ defined over $\bar{k}$. However all of these are given by $\mathrm{GL}_n(\Z)$ and defined over $k$, so this map commutes with the action of $G_k$. Hence for any $\sigma \in G_k$, we have the following equality of maps $(A_{\bar{k}})^n \rightarrow (A_{\bar{k}})^n$:
    \begin{equation*}
        \begin{split}
            s_{N \otimes A}(\sigma)\psi_{N \otimes A}^{-1}\phi_A \psi_{M \otimes A} & = \psi_{N \otimes A}^{-1} \sigma \psi_{N \otimes A} \sigma^{-1} \psi_{N \otimes A}^{-1} \phi_A \psi_{M \otimes A} \\ & = \psi_{N \otimes A}^{-1} \sigma \phi_A \psi_{M \otimes A} \sigma^{-1} \\ & = \psi_{N \otimes A}^{-1} \phi_A \sigma \psi_{M \otimes A} \sigma^{-1} \\ & = \psi_{N \otimes A}^{-1} \phi_A \psi_{M \otimes A}s_{M \otimes A}(\sigma),
        \end{split}
    \end{equation*}
    where the second equality holds because $\sigma$ commutes with $\psi_{N \otimes A}^{-1} \phi_A \psi_{M \otimes A}$, and the third because it commutes with $\phi_A$.

    Now apply $\rho^{-1}$ to this to get 
    $$s_N(\sigma)\psi_N^{-1}\phi\psi_M = \psi_N^{-1}\phi \psi_M s_M(\sigma).$$ Because Galois acts trivially on $\Z^n$ this implies $$\psi_N^{-1}\sigma\phi\psi_M = \psi_N^{-1}\phi\sigma\psi_M$$ which tells us that $\phi$ commutes with $\sigma$. Hence $\phi$ is an isomorphism of $\Z[G_k]$-modules.
\end{proof}

\begin{remark}
    We can also give a proof of this in terms of cohomology. As in (\cite{Milne} \S2), the twists $A \otimes M$ and $A \otimes N$ of $A^n$ correspond to cohomology classes $c_{A \otimes M}$ and $c_{A \otimes N}$ in $H^1(G_k, \mathrm{Aut}((A_{\bar{k}})^n))$. If the twists are isomorphic over $k$, $c_{A \otimes M} = c_{A \otimes N}$. Similarly, the modules $M$ and $N$ are twists of the trivial $G_k$-module $\Z^n$, and define cohomology classes $c_M$ and $c_N$ in $H^1(G_k, \mathrm{Aut}(\Z^n))$. The map $\rho$ identifies $\mathrm{Aut}(\Z^n)$ and $\mathrm{Aut}((A_{\bar{k}})^n)$ as groups, and also as $G_k$-modules because all of the elements of $\mathrm{Aut}((A_{\bar{k}})^n)$ are defined over $k$, so the action of $G_k$ on both groups is trivial. By construction, under this isomorphism the classes $c_{A \otimes M}$ and $c_{A \otimes N}$ correspond to $c_M$ and $c_N$ respectively, so $c_M=c_N$. Therefore $M$ and $N$ are isomorphic as $G_k$-modules.
\end{remark}

\begin{proof}[Proof of Theorem \ref{secondtheorem}]
    Let $M$ and $N$ be two ideals of $\Z[\zeta_{23}]$ in different ideal classes (23 can be replaced by any prime such that $\Z[\zeta_{23}]$ has non-trivial class group). Then, by Lemmas \ref{lemma7} and \ref{lemma8}, they are not isomorphic as $\Z[C_{23}]$-modules but $M \otimes \Z_{\ell} \cong N \otimes \Z_{\ell}$ for all primes $\ell$. This implies $M \otimes \Z_{(\ell)} \cong N \otimes \Z_{(\ell)}$ (\cite{CR} Cor. 76.9). Therefore there is an injective homomorphism $M \rightarrow N$ with finite cokernel of order coprime to $\ell$.

    Now pick an elliptic curve $E$ over $\mathbb{Q}$, with no potential complex multiplication. Pick a number field $L$ with $\mathrm{Gal}(L/\Q) \cong C_{23}$, and let $G_{\Q}$ act on $M$ and $N$ via the corresponding $C_{23}$ quotient. Then $A := M \otimes E$ and $B := N \otimes E$ are related by an isogeny of degree coprime to $\ell$ for each $\ell$ (by Lemma \ref{milne}), but are not isomorphic (by Lemma \ref{milneconverse}).
\end{proof}
\begin{remark}
    $M \otimes E$ and $N \otimes E$ are defined over $\mathbb{Q}$, and isomorphic over $\bar{\Q}$ to $E^{22}$. In fact, they are isomorphic over $L$, because $M$ and $N$ are isomorphic as $\Z[G_L]$-modules, where $G_L$ acts trivially.
\end{remark}

\begin{remark}
    For elliptic curves over $\Q$, the Tate modules determine the curve up to isomorphism. This follows from the fact that $\mathrm{Hom}_{\Z_p[G_{k}]}(T_p(E),T_p(E')) \cong \Z_p \otimes \mathrm{Hom}(E,E')$ (\cite{Faltings} \S 5, Cor. 1). Indeed, if the Tate modules $T_p(E)$ and $T_p(E')$ are isomorphic, then $E$ and $E'$ are isogenous, and over $\Q$ we must have $\mathrm{Hom}_\Q(E,E') \cong \Z$ (if it were larger we would have complex multiplication defined over $\Q$). Let this be generated by $\phi$. Now pick $\Z_p$-bases for the Tate modules, and consider the determinant of the $\Z_p$-linear map on $T_p(E)$ induced by $\phi$. While the determinant depends on the choice of bases, its $p$-adic valuation does not. If the degree of $\phi$ is divisible by $p$, then this determinant is also divisible by $p$. Hence the same is true for the maps induced by every element of $\Z_p \otimes \mathrm{Hom}(E,E')$, so none of them gives an isomorphism of Tate modules, and we must have $p \nmid \mathrm{deg}(\phi)$. If this holds for all $p$, then $\phi: E \rightarrow E'$ is an isomorphism.
\end{remark}

\subsection*{Acknowledgements}
I would like to thank my supervisor Vladimir Dokchitser for suggesting this problem, and his guidance in solving it. I would also like to thank Dominik Bullach for his helpful discussions about integral representation theory, and the referees for their comments on the paper.

This work was supported by the Engineering and Physical Sciences Research Council [EP/L015234/1], the EPSRC Centre for Doctoral Training in Geometry and Number Theory (The London School of Geometry and Number Theory) at University College London.

\end{document}